\documentclass[12pt,reqno]{amsart}
\usepackage{fullpage}

\newtheorem{theorem}{Theorem}[section]
\newtheorem{lemma}[theorem]{Lemma}
\newtheorem{prop}[theorem]{Proposition}
\newtheorem{cor}[theorem]{Corollary}

\usepackage{graphicx}
\usepackage{color}
\usepackage{subfigure}
\usepackage{amssymb}
\usepackage{amsmath,mathrsfs}
\usepackage{colonequals}
\usepackage{hyperref}

\theoremstyle{definition}
\newtheorem{definition}[theorem]{Definition}

\newtheorem{assumption}{Assumption}

\renewcommand{\subset}{\subseteq}
\renewcommand{\supset}{\supseteq}
\renewcommand{\epsilon}{\varepsilon}
\renewcommand{\nu}{v}

\newcommand{\ignore}[1]{}
\newcommand{\abs}[1]{\left|#1\right|}                   
\newcommand{\absf}[1]{|#1|}                             
\newcommand{\vnorm}[1]{\left\|#1\right\|}    

\newcommand{\Z}{\mathbb{Z}}                             

\newcommand{\E}{\mathbb{E}}

\newcommand{\R}{\mathbb{R}}
\newcommand{\C}{\mathbb{C}}

\newcommand{\eps}{\varepsilon}

\newcommand{\embolden}[1]{\textbf {#1}}

\newcommand{\flatsub}{\substack{\{B_{i}\}_{i=1}^{n+1}\,
\mathrm{is}\,\mathrm{a}\,\mathrm{flat}\,\mathrm{partition}\,\mathrm{of}\,\R^{n}\\ \mathrm{with}\,\mathrm{volumes}\,(a_{1},\ldots,a_{n+1})}}
\newcommand{\partsub}{\substack{\{B_{i}\}_{i=1}^{n+1}\,
\mathrm{is}\,\mathrm{a}\,\mathrm{partition}\,\mathrm{of}\,\R^{n}\\ \mathrm{with}\,\mathrm{volumes}\,(a_{1},\ldots,a_{n+1})}}
\newcommand{\partsubpin}{\substack{\{B_{p}\}_{p=1}^{n+1}\,
\mathrm{is}\,\mathrm{a}\,\mathrm{partition}\,\mathrm{of}\,\R^{n}\\ \mathrm{with}\,\mathrm{volumes}\,(a_{1},\ldots,a_{n+1})}}

\begin{document}

\title{Standard Simplices and Pluralities are Not the Most Noise Stable}

\thanks{S. H. was supported by NSF Graduate Research Fellowship DGE-0813964 and a Simons-Berkeley Research Fellowship.
Part of this work was completed while S. H. visited the Network Science and Graph Algorithms program at ICERM.  E. M. was supported by NSF grant DMS-1106999, NSF Grant CCF 1320105 and DOD ONR grant N000141110140.  J. N. was supported by NSF grant DMS-1106999 and DOD ONR grant N000141110140.  Part of this work was carried out while the authors were visiting the Real Analysis in Computer Science program at the Simons Institute for the Theory of Computing.}

\author{Steven Heilman}
\address{UCLA Department of Mathematics, Los Angeles, CA 90095-1555}
\email{heilman@math.ucla.edu}
\author{Elchanan Mossel}
\address{Statistics and Computer Science, University of California, Berkeley, CA  94720}
\email{mossel@stat.berkeley.edu}
\author{Joe Neeman}
\address{Mathematics and Electrical and Computer Engineering, University of Texas at Austin, Austin, TX 78712}
\email{joeneeman@gmail.com}

\date{\today}
\begin{abstract}
The Standard Simplex Conjecture and the Plurality is Stablest Conjecture are two conjectures stating that certain partitions are optimal with respect to Gaussian and discrete noise stability respectively.
These two conjectures are natural generalizations of the Gaussian noise stability result by Borell (1985) and the Majority is Stablest Theorem (2004). Here we show that the standard simplex is not the most stable partition in Gaussian space and that Plurality is not the most stable low influence partition in discrete space
for every number of parts $k \geq 3$, for every value $\rho \neq 0$ of the noise and for every prescribed measures for the different parts as long as they are not all equal to $1/k$.
Our results do not contradict the original statements of the Plurality is Stablest and Standard Simplex Conjectures in their original statements concerning partitions to sets of equal measure.
However, they indicate that if these conjectures are true, their veracity and their proofs will crucially rely on assuming that the sets are of equal measures, in stark contrast to Borell's result, the Majority is Stablest Theorem and many other results in isoperimetric theory.
Given our results it is natural to ask for (conjectured) partitions achieving the optimum noise stability.

\end{abstract}
\maketitle

\section{Introduction}

{\em Noise stability} is a natural concept which appears in the study of Gaussian processes, voting, percolation and theoretical computer science. The study of partitions of the space which are optimal with respect to noise stability may be viewed as a natural extension of isoperimetric theory; see e.g.~\cite[Chapter 8]{ledoux96}.

The basic case which was studied most extensively is the following: which partitions of a space into two parts (with given measures) maximize noise stability? The answer to this question follows the development of isoperimetric theory:  the surface-minimizing body in $\R^n$ with prescribed Lebesgue measure is a ball ~\cite{steiner38,wei67,levy51}.  In the sphere it is a cap, or geodesic ball ~\cite{levy51}  which in turn implies that in Gaussian space it is a half-space ~\cite{borell75,sudakov74,borell03}.  The answer to the noise stability question is analogous: half-spaces maximize the Gaussian noise stability among all sets of a given measure~\cite{borell85,mossel12,eldan13}.  Using an invariance principle, this implies that the majority functions maximize noise stability among all low-influence functions on the discrete cube~\cite{mossel10}. Informally, a common feature of all of the results above is that the geometric nature of the optimal partition remains the same no matter what the prescribed measure is (a half-space, a cap etc.).

\subsection{Partitions into more than two parts}
A much more challenging question deals with partitions into more than two parts. For the isoperimetric question,
it took more than $100$ years to prove the ``Double Bubble Theorem''~\cite{hutchings02} which determines the minimal surface area that encloses and separates two fixed volumes in
$\R^3$. The optimal partition is given by two spheres which intersect at a $120^{\circ}$ angle having a third spherical cap separating the two volumes. Such a partition is called a double bubble partition. It is further conjectured that multi-bubble partitions minimize the surface area when partitioning into $k$ parts in $\R^n$ as long as $k \leq n+1$~\cite[Proposition 2]{sullivan96}~\cite[p. 153]{morgan09}.
An analogous isoperimetric result was established in Gaussian space, i.e. $\R^n$ equipped
with the standard Gaussian density.  Building on the previous work on the Double Bubble Conjecture,
the authors of~\cite{corneli08} found the partition of $\R^n$ $(n \geq 2)$ into three parts,
each having Gaussian volume about~$\frac{1}{3}$, that minimizes the Gaussian surface area between the three volumes. Their work shows that the optimal partition is a standard simplex
partition, which can be seen as the limit of the double bubble partition scaled up around one point on the intersection.  The standard simplex partition is defined by taking a regular simplex with center $P$, so that each element of the partition is the cone of a facet of the simplex with common base point $P$.

The analogies with the isoperimetric problem led the authors of~\cite{isaksson11} to conjecture that standard simplices are optimal for Gaussian noise stability.

In discrete product spaces, the questions regarding noise stability are central in studying hardness of approximation and voting. The relevant question in this setting involves low-influence partitions.
Using a non-linear invariance principle it was shown that the majority function maximizes noise stability~\cite{mossel10} as conjectured both in the context of hardness of approximation~\cite{khot07} and in voting~\cite{kalai02}.
Since plurality is the natural generalization of majority, it was conjectured in~\cite{khot07} that  plurality is the most stable low-influence partition into $k \geq 3$ parts of $\{1,\ldots,k\}^n$. Furthermore it was shown in~\cite{isaksson11} that the standard simplex being optimal for Gaussian noise stability is equivalent to the fact that plurality is stablest.
Therefore, the authors of~\cite{khot07,isaksson11} conjectured that plurality is stablest and demonstrated a number of applications of this result in hardness of approximation and voting.

\subsection{Our Results}
In our main results we show that {\em the standard simplex is not the most stable partition in Gaussian space and that plurality is not the most stable low-influence partition in discrete space}
for every number of parts $k \geq 3$, for every value $\rho \neq 0$ of the noise stability and for every prescribed measures for the different parts as long as they are not all equal to $1/k$.

In other words, the optimal partitions for noise stability are of a different nature than the ones considered for partitions into three parts in isoperimetric theory. Thus, we now know that the extension of noise stability theory from two to three or more parts is very much different than the extension of isoperimetric theory from two to three or more parts.  Moreover, all existing proofs which optimize noise stability of two sets \cite{borell85,kindler12,isaksson11,mossel12,eldan13} must fail for more than three sets, since these proofs rely on the fact that a half-space optimizes noise stability with respect to any measure restriction.

\subsection{Proof Techniques}
The main new ingredient in our work is developed in Section \S\ref{secanalytic}, in particular, Lemma \ref{lemma4}.  Here it is shown that the Ornstein-Uhlenbeck operator applied to the indicator function of a simplicial cone becomes holomorphic when restricted to certain lines.  This holomorphicity condition, when combined with a first variation argument (i.e. an infinite dimensional perturbative argument of the first order), then shows that any simplicial cone can be perturbed in a volume-preserving manner to improve its noise stability.  Such a holomorphicity argument seems unavailable for the isoperimetric problem, since this argument uses the inherent nonlocality of the Ornstein-Uhlenbeck semigroup.  That is, this semigroup at a point $x$ computes an average over all points in Euclidean space.  On the other hand, the local arguments used in isoperimetric theory, e.g. in the investigation of solutions of mean curvature flow \cite{colding11,colding12,colding12a}, do not seem to be available in the noise stability setting.  More specifically, the area of a surface is equal to the sum of the areas of many small pieces of the that surface, so it suffices to change a small portion of that surface to improve the total surface area.  Yet, the noise stability of a set is not equal to the sum of the noise stabilities of many small pieces of the set, so the same logic does not apply in our setting.  From this perspective, the holomorphicity argument could be surprising, since it allows for a local perturbation argument for the nonlocal quantity of the noise stability.

An additional argument we use is the
the first variation argument within Lemma \ref{lemma1.1} which is more standard.
For readers who may not be very familiar with variational arguments, we recall the basic setup and provide some references.
When there is some quantity to minimize involving integrals, one wants to say that the function or set minimizing the integral has a zero derivative, in some sense.  Often the space of functions or sets  is infinite dimensional, so one has to be careful with what it means for the optimizer to have zero derivatives.

In our context, the notion of zero derivative of noise stability is then formalized using the first variation, as follows.  Let $A$ be a set of fixed volume and optimal noise stability, and let $\{A^{(t)}\}_{t\in(-1,1)}$ be a family of sets such that $A^{(0)}=A$ and such that the volume of $A^{(t)}$ is equal to the volume of $A$.  Then the rate of change of the noise stability of $A^{(t)}$ must be zero at $t=0$, assuming that the sets $A^{(t)}$ change in a smooth manner as $t$ changes.  We employ this general principle by starting with a set $A$ with smooth boundary, and then by smoothly perturbing the boundary of $A$ to maintain the desired volume constraint.  This argument is known as a normal variation argument, and it is standard within the literature of isoperimetric inequalities or of the calculus of variations.  For just a few of many possible references which are relevant here, see~\cite{colding11a,chokski07,chavel06}.  In fact, we are dealing with the noise stability of multiple sets, and within Lemma \ref{lemma1.1} we perturb two sets $A_{i},A_{j}$ simultaneously, while leaving the remaining sets intact.  Yet, the argument is similar to the case of one set, since we restrict our attention to sets which are simplicial cones.

\subsection{The Original Conjectures}
We note again that our results do not contradict the precise conjectures stated in~\cite{khot07,isaksson11} as these are stated in the case where the measures of all partition elements are exactly $1/k$. However, the authors of~\cite{khot07,isaksson11} gave no indication that they believe there is something special about the case of equal measured partitions -- rather, they made the conjectures that were needed for the applications presented in their paper.

\section{Definition and Statement of Main Results}
\subsection{Gaussian Noise Stability}


Let $n\geq1$, let $y=(y_{1},\ldots,y_{n})\in\R^{n}$, and define $d\gamma_{n}(y)\colonequals e^{-(y_{1}^{2}+\cdots+y_{n}^{2})/2}dy/(2\pi)^{n/2}$.   Let $\ell_{2}^{n}$ denote the $\ell_{2}$ norm on $\R^{n}$.  For $r>0$, define $B(y,r)\colonequals\{x\in\R^{n}\colon\vnorm{x-y}_{\ell_{2}^{n}}<r\}$.

\begin{definition}\label{OUDef}
Let $n\geq1$, $n\in\Z$.  Let $\rho\in(-1,1)$, $x\in\R^{n}$.  Let $f\colon\R^{n}\to[0,1]$.  The Ornstein-Uhlenbeck (or Bonami-Beckner, or noise) operator is defined by
$$T_{\rho}f(x)\colonequals\int_{\R^{n}}f(x\rho+y\sqrt{1-\rho^{2}}\,)d\gamma_{n}(y).$$
\end{definition}

\begin{definition}\label{partdef}
Let $A_{1},\ldots,A_{k}\subset\R^{n}$ be measurable, $k\leq n+1$.  We say that $\{A_{i}\}_{i=1}^{k}$ is a \textit{partition} of $\R^{n}$ if $\cup_{i=1}^{k}A_{i}=\R^{n}$, and $\gamma_{n}(A_{i}\cap A_{j})=0$ for all $i\neq j$, $i,j\in\{1,\ldots,k\}$.
\end{definition}


\begin{definition}\label{measdef}
Let $\{A_{i}\}_{i=1}^{k}$ be a partition of $\R^{n}$.  We say that $\{A_{i}\}_{i=1}^{k}$ \textit{has volumes} $(a_{1},\ldots,a_{k})$ if $\gamma_{n}(A_{i})=a_{i}$ for all $i=1,\ldots,k$.
\end{definition}

\begin{definition}[\embolden{Flat / Simplex Partitions}]\label{flatdef}
Let $\{A_{i}\}_{i=1}^{k}$ be a partition of $\R^{n}$.  We say that $\{A_{i}\}_{i=1}^{k}$ is a \textit{flat partition} if there exist  $y\in\R^{n}$ and $\{y_{i}\}_{i=1}^{k}\subset\R^{n}\setminus\{0\}$, such that
\begin{itemize}
\item
For all $i,j\in\{1,\ldots,k\}$, $i\neq j$, $y_i$ is not a positive multiple of $y_{j}$, and
\item
for all $i\in\{1,\ldots,k\}$, $A_i = y + \{x\in\R^{n}\colon\langle x,y_{i}\rangle=\max_{j=1,\ldots,k}\langle x,y_{j}\rangle\}$.
\end{itemize}
If $y=0$ we say that $\{A_{i}\}_{i=1}^{k}$ is \textit{centered}. If $y\neq0$, we say it is \textit{shifted}.

A \textit{standard simplex partition} is a flat partition where $\| y_i \|_2 = 1$ for all $i$ and $\langle y_i, y_j \rangle = -\frac{1}{k-1}$ for all $i \neq j$, $i,j\in\{1,\ldots,k\}$. Again, we call a standard simplex partition centered if $y=0$ and shifted if $y\neq0$.
\end{definition}


\begin{definition}[\embolden{Gaussian Noise Stability}]\label{noisedef}
The Gaussian noise stability of a partition $\{A_{i}\}_{i=1}^{k}$ is given by
\[
S_{\rho}(\{A_{i}\}_{i=1}^{k})
\colonequals\sum_{i=1}^{k}\int_{\R^{n}}1_{A_{i}}T_{\rho}1_{A_{i}}d\gamma_{n}.
\]
\end{definition}

Note the following probabilistic interpretation of noise stability.  Let $X=(X_{1},\ldots,X_{n})$ and let $Y=(Y_{1},\ldots,Y_{n})\in\R^{n}$ be standard Gaussian random vectors such that $\E(X_{i}Y_{j})=\rho1_{i=j}$ for all $i,j\in\{1,\ldots,n\}$.  Then for any $A\subset\R^{n}$, $\int 1_{A}T_{\rho}1_{A}d\gamma_{n}=\mathrm{Prob}((X,Y)\in A\times A)$.  So, Definition~\ref{noisedef} can be equivalently written as $\sum_{i=1}^{k}\mathrm{Prob}((X,Y)\in(A_{i},A_{i}))$.

\subsection{Optimal Partitions are not Flat}

As mentioned earlier, the result of \cite{corneli08} finds the partition of $\R^{n}$ into three sets of fixed Gaussian measures $a_{1},a_{2},a_{3}$ with $\abs{a_{i}-1/3}<.04$ for all $i=1,2,3$ and of minimum total Gaussian surface area.  This partition is always given by three $120$ degree sectors, so if the $a_i$ are not all $1/3$ then it is a shifted standard simplex partition.
This, along with the standard simplex conjecture of  \cite{mossel10} may suggest that standard simplex partitions are always optimal.
In our main result we prove that this is not the case:

\begin{theorem}\label{thm1}
Let $\{A_{i}\}_{i=1}^{n+1}\subset\R^{n}$ be a shifted flat partition with volumes $a\colonequals(a_{1},\ldots,a_{n+1})\neq(1/(n+1),\ldots,1/(n+1))$.  Let $0<\rho<1$. Then
\begin{equation}\label{zero0}
S_{\rho}(\{A_{i}\}_{i=1}^{n+1}) <
\sup_{\partsub}S_{\rho}(\{B_{i}\}_{i=1}^{n+1}).
\end{equation}
where the supremum is taken over all partitions.
Similarly for $-1<\rho<0$, we have that:
\begin{equation}\label{zero1.1}
S_{\rho}(\{A_{i}\}_{i=1}^{n+1}) >
\inf_{\partsub}S_{\rho}(\{B_{i}\}_{i=1}^{n+1}).
\end{equation}
 \end{theorem}

Our results indicate in particular that there is no straightforward generalization of the any of the methods of \cite{borell85,kindler12,isaksson11,mossel12,eldan13} that is applicable to the study of noise stability for more than two sets.  As mentioned earlier, Theorem~\ref{thm1} also indicates a dramatic difference between the noise stability problem with two sets (where the optimal partitions are always simplex partitions) and three or more sets. This also indicates a difference between the isoperimetric problem with three sets (where in all the known cases the partitions are simplex partitions) and the noise stability problem.

\subsection{Discrete Noise Stability}
Using the invariance principle~\cite{mossel10}, \cite[Theorem 1.10]{isaksson11}, it is by now standard to deduce discrete analogues of our main theorem.
For simplicity we formulate one special case for partitions into $3$ parts with $0<\rho<1$ in Corollary~\ref{cor:plu} below, which only requires the Central Limit Theorem.
\begin{definition}[\embolden{Plurality}]\label{plurdef}
 A {\em plurality function} $f_{n} : \{1,2,3\}^n \to \{1,2,3\}$ is a function satisfying $f_{n}(x_1,\ldots,x_n) \colonequals i$ whenever
$ \abs{\{ \ell\in\{1,\ldots,n\} : x_\ell = i \}} > \abs{\{ \ell\in\{1,\ldots,n\} : x_\ell = j \}}$ for all $j \neq i$, $i,j\in\{1,2,3\}$. Note that we do not prescribe any behavior in cases of ties, i.e.~when $\exists$ $i,j\in\{1,2,3\}$ such that $\abs{\{ \ell\in\{1,\ldots,n\} : x_\ell = i \}} =\abs{\{ \ell\in\{1,\ldots,n\} : x_\ell = j \}}$.
\end{definition}
Since the values of $f_n$ for ties does not concern us particularly, we will often write ``plurality''
or ``the plurality function'' to denote an arbitrary plurality function.

\begin{definition}[\embolden{Influences}]\label{infdef}
Given a probability measure $P$ on $ \{1,2,3\}$, the $i$'th
 {\em influence} $I_i(f)\in\R$ of a function $f : \{1,2,3\}^n \to \{1,2,3\}$ is defined by
 \[
 I_i(f) := P[f(x_1,\ldots,x_{i-1},x_i,x_{i+1},\ldots,x_n) \neq f(x_1,\ldots,x_{i-1},y,x_{i+1},\ldots,x_n)],
 \]
 where $x_i,y$ are i.i.d. samples from $P$ for all $i=1,\ldots,n$. We say that the sequence $f_n\colon\{1,2,3\}^n \to \{1,2,3\} $ has {\em vanishing influences} if
 \[
 \lim_{n \to \infty} \max_{1 \leq i \leq n} I_i(f_n) = 0.
 \]
\end{definition}
Note that any sequence of plurality functions has vanishing influences.

\begin{cor}[\embolden{Plurality is not Stablest}] \label{cor:plu}
Let $\alpha,\beta\in\R$.  Consider the space $\{1,2,3\}$ equipped with the probability measure $\tilde{Q}$ where
\[
\tilde{Q}(1) = \frac{1}{3} + \alpha n^{-1/2}, \quad \tilde{Q}(2) = \frac{1}{3} + \beta n^{-1/2}, \quad \tilde{Q}(3) = \frac{1}{3} -(\alpha+\beta) n^{-1/2}.
\]
Let $\tilde{P}$ be the distribution of correlated pairs on $\{1,2,3\}^2$ with marginals $\tilde{Q}$ where the two values agree with probability $\rho>0$ and are otherwise independent: for all $x,y\in\{1,2,3\}$,
\[
\tilde{P}(x,y) = \rho1_{(x=y)} \tilde{Q}(x) + (1-\rho) \tilde{Q}(x) \tilde{Q}(y),
\]
and let $P =  \tilde{P}^n$.
Let $f_n :\{1,2,3\}^n \to \{1,2,3\}$ be the sequence of plurality functions and assume $(\alpha,\beta) \neq (0,0)$.
Then there exist a sequence $g_n\colon\{1,2,3\}^{n}\to\{1,2,3\}$ of functions of vanishing influences with respect to $P$ such that
\[
\lim_{n \to \infty} P[f_n = i] =  \lim_{n \to \infty} P[g_n = i], \quad \forall i \in\{1,2,3\}
\]
and for $x,y\in\{1,2,3\}^{n}$ distributed according to $P$,
\[
\lim_{n \to \infty} P[f_n(x) = f_n(y)] < \lim_{n \to \infty} P[g_n(x) = g_n(y)].
\]
\end{cor}
Informally, the corollary says that plurality is not the most stable low-influence function among all functions that have the same probability of getting all outcomes, when the election is slightly biased.  Here we think of $n$ as the number of voters in an election, $\{1,2,3\}$ as possible votes between $3$ candidates, and $f_{n},g_{n}$ as functions which input the votes and output the winner of the election. We further note that for every fixed values of $\alpha$ and $\beta$ a sequence of functions has vanishing influences if and only if it has vanishing influences under the uniform measure.

\subsection{Bilinear Noise Stability}
Given our main results, it is natural to ask {\em what is} optimal partition for noise stability when partitioning into three or more parts of unequal measures.
While we do not have a conjecture for what the partition is, there is a very special case where we consider two different partitions of special sizes where the optimal partitions can be found explicitly using Borell's results.


\begin{prop}\label{q2}
Let $\rho\in(0,1)$.  Then there exist partitions $\{A_{1},A_{2},A_{3}\}\subset\R^{2}$ and $\{B_{1},B_{2},B_{3}\}\subset\R^{2}$ with volumes $a=(1/3,1/3,1/3)$ and $b=(1/2,0,1/2)$ respectively, so that
\begin{equation}\label{zero5}
\sum_{i=1}^{3}\int_{\R^{2}}1_{A_{i}(a)}T_{\rho}1_{B_{i}(b)}d\gamma_{2}{=}
\sup_{\substack{\{C_{i}\}_{i=1}^{3}\,
\mathrm{is}\,\mathrm{a}\,\mathrm{partition}\,\mathrm{of}\,\R^{2}\,\mathrm{with}\,\mathrm{volumes}\,a\\
\{D_{i}\}_{i=1}^{3}\,
\mathrm{is}\,\mathrm{a}\,\mathrm{partition}\,\mathrm{of}\,\R^{2}\,\mathrm{with}\,\mathrm{volumes}\,b}}
\sum_{i=1}^{3}\int_{\R^{2}}1_{C_{i}}T_{\rho}1_{D_{i}}d\gamma_{2}.
\end{equation}
Moreover the maximizing partitions (must) satisfy that $A_1$ and $B_1$ (and also $A_3$ and $B_3$) are parallel half-spaces of measures $1/3$ and $1/2$.
\end{prop}
Note in particular that the partition $\{A_{1},A_{2},A_{3}\}$ is not a standard simplex partition. The proposition in this case follows from Borell's results as each of the terms
$\int_{\R^{2}} f_1 T g_1 d\gamma_{2}$ and $\int_{\R^{2}} f_3 T g_3 d\gamma_{2}$ is individually maximized by $A_1,B_1$ and $A_3, B_3$ respectively. Moreover, given the uniqueness version of Borell's result~\cite{mossel12,eldan13} these are the only partitions that achieve the maximum.
We note again that in the case of partitions into two parts, even in the bilinear case of partitions $\{A_1,A_2\}$ and $\{B_1,B_2\}$ the optimal partitions are always half-space (standard simplex) partitions.

We further note that even if the measures are slightly perturbed for example by letting
$\tilde{b} = (1/2-\eps, 2 \eps,1/2-\eps)$ for a small $\eps>0$ then any (almost) optimal partitions $\tilde{A}, \tilde{B}$ for $a,\tilde{b}$ in \eqref{zero5} have to be close to those given by $A$ and $B$ in the sense that they can be matched so that
$\gamma_n(\tilde{A}_i \Delta A_i)$ and $\gamma_n(\tilde{B}_i \Delta B_i)$ are small for all $1 \leq i \leq 3$.
This follows since if $A,B$ are almost optimal for $a,\tilde{b}$ then they are almost optimal for $a,b$.
The robust versions of Borell's result given in~\cite{mossel12,eldan13} then implies that $A,B$ have to be close to the optimal partition.

\subsection{Organization}

In \S\ref{secvar}, we compute the first variation of the noise stability of flat partitions.  This computation shows that, if a partition $\{A_{i}\}_{i=1}^{n+1}$ optimizes noise stability, then $T_{\rho}(1_{A_{i}}-1_{A_{j}})$ must be constant on $A_{i}\cap A_{j}$ for all $i,j\in\{1,\ldots,n+1\}$, $i\neq j$.  We will use this statement in the contrapositive form, since in \S\ref{secanalytic}, we show that for many flat partitions $\{A_{i}\}_{i=1}^{n+1}$, there exists $i,j\in\{1,\ldots,n+1\}$ with $i\neq j$ such that $T_{\rho}(1_{A_{i}}-1_{A_{j}})$ is not constant on $A_{i}\cap A_{j}$.  And therefore, many flat partitions $\{A_{i}\}_{i=1}^{n+1}$ do not optimize noise stability.  The Lemmas of \S\ref{secvar} and \S\ref{secanalytic} are combined in \S\ref{secmain}, where Theorem~\ref{thm1} is proved.  Finally, Corollary~\ref{cor:plu} is shown to be a consequence of Theorem~\ref{thm1} in \S\ref{secvoting}.

\section{The First Variation}\label{secvar}


We compute the first variation of the noise stability, in order to understand the sets that achieve equality in \eqref{zero0} and \eqref{zero5}.  The following argument is a normal variation implementation of arguments from \cite[Lemma 3.3]{khot09}, \cite[Lemma 2.7]{heilman12}.  Basically, we move some mass between two sets to show that the noise stability can be improved.  The movement of mass is defined by a vector field $V$.

\begin{lemma}[\embolden{Normal Variation}]\label{lemma1.1}
Let $\rho\in(-1,0)\cup(0,1)$.  Let $\{A_{i}\}_{i=1}^{n+1}$ be a flat partition of $\R^{n}$ with volumes $(a_{1},\ldots,a_{n+1})$.  Assume that there exist $i,j\in\{1,\ldots,n+1\}$ with $i\neq j$ such that $T_{\rho}(1_{A_{i}}-1_{A_{j}})$ is not constant on $A_{i}\cap A_{j}$.  Then
\begin{equation}\label{one1}
\begin{aligned}
\rho\in(0,1)&\quad\Longrightarrow\quad
S_{\rho}(\{A_{i}\}_{i=1}^{n+1})
<\sup_{\partsubpin}S_{\rho}(\{B_{i}\}_{i=1}^{n+1}).\\
\rho\in(-1,0)&\quad\Longrightarrow\quad
S_{\rho}(\{A_{i}\}_{i=1}^{n+1})
>\inf_{\partsubpin}S_{\rho}(\{B_{i}\}_{i=1}^{n+1}).
\end{aligned}
\end{equation}
\end{lemma}
\begin{proof}
Let $A\subset\R^{n}$ be a set with smooth boundary, and let $N\colon\partial A\to S^{n-1}$ denote the unit exterior normal to $\partial A$.  Let $V\colon\R^{n}\to \R^{n}$ be specified later.  Let $\Psi\colon\R^{n}\times(-1,1)$ such that $(d/dt)|_{t=0}\Psi(x,t)=V(x)$ for all $x\in\R^{n}$, and such that $\Psi(x,0)=x$ for all $x\in\R^{n}$.
For $x,y\in\R^{n}$, define
$$G(x,y)\colonequals e^{-\vnorm{\rho x-y}_{2}^{2}/[2(1-\rho^{2})]}.$$
Define $A^{(t)}\colonequals\Psi(A,t)$.  For $f\colon\R^{n}\to\R$, using Definition \ref{OUDef}, write
\begin{equation}\label{one5.0}
T_{\rho}f(x)\equalscolon\int f(y)G(x,y)dy.
\end{equation}

Let $J\Psi(y,t)$ denote $\abs{\det(\partial\Psi_{i}(y,t)/\partial y_{j})_{1\leq i,j\leq n}}\in\R$.  We claim that
\begin{equation}\label{one5}
\frac{d}{dt}|_{t=0}T_{\rho}1_{A^{(t)}}(x)=\int_{\partial A} G(x,y)\langle V(y),N(y)\rangle dy.
\end{equation}
Indeed, using that $(d/dt)|_{t=0}\Psi(x,t)=V(x)$ and $\Psi(x,0)=x$, $J\Psi(y,0)=1$ for all $y$, so the chain rule and divergence theorem show that
\begin{equation}\label{one6}
\begin{aligned}
\frac{d}{dt}|_{t=0}\int 1_{A^{(t)}}(y)G(x,y)dy
&=\frac{d}{dt}|_{t=0}\int_{A^{(t)}}G(x,y)dy
=\frac{d}{dt}|_{t=0}\int_{A}G(x,\Psi(y,t))J\Psi(y,t)dy\\
&=\int_{A}\mathrm{div}_{y}(G(x,y)V(y))dy
=\int_{\partial A}G(x,y)\langle V(y),N(y)\rangle dy.
\end{aligned}
\end{equation}
So, using $G(x,y)\gamma_{n}(x)=G(y,x)\gamma_{n}(y)$,
\begin{equation}\label{one7}
\begin{aligned}
&\frac{d}{dt}|_{t=0}\int 1_{A^{(t)}}T_{\rho}1_{A^{(t)}}d\gamma_{n}
=2\int 1_{A}(\frac{d}{dt}|_{t=0}T_{\rho}1_{A^{(t)}})d\gamma_{n}\\
&\quad\stackrel{\eqref{one5}}{=}2\int_{A}\int_{\partial A}G(x,y)\langle V(y),N(y)\rangle dy d\gamma_{n}(x)\\
&\quad=2\int_{\partial A}\int_{A}G(x,y)d\gamma_{n}(x)\langle V(y),N(y)\rangle dy
=2\int_{\partial A}\int_{A}G(y,x)dx\langle V(y),N(y)\rangle \gamma_{n}(y)dy\\
&\quad\stackrel{\eqref{one5.0}}{=}2\int_{\partial A}T_{\rho}1_{A}(y)\langle V(y),N(y)\rangle \gamma_{n}(y)dy.
\end{aligned}
\end{equation}

Suppose $V\colon\R^{n}\to\R^{n}$ satisfies
\begin{equation}\label{one7.1}
\int_{\partial A}\langle V(y),N(y)\rangle \gamma_{n}(y)dy=0.
\end{equation}
 We then claim that $\Psi$ is volume-preserving at $t=0$.  Indeed, using \eqref{one5} for $\rho=0$,
 $$\frac{d}{dt}|_{t=0}\gamma_{n}(A^{(t)})=\int_{\partial A}\langle V(y),N(y)\rangle d\gamma_{n}(y).$$
 %


Let $N$ denote the exterior normal to $A_{i}$, so that $-N$ is the exterior normal to $A_{j}$.  For a subset of $(\partial A_{i})\cap(\partial A_{j})$, we denote $\gamma_{n-1}$ as the measure $\gamma_{n}$ restricted to $(\partial A_{i})\cap(\partial A_{j})$.  Since $T_{\rho}(1_{A_{i}}-1_{A_{j}})$ is continuous and not constant on $A_{i}\cap A_{j}$, there exist disjoint relatively open sets $U_{1},U_{2}\subset(\partial A_{i})\cap(\partial A_{j})$ such that $\gamma_{n-1}(U_{1})=\gamma_{n-1}(U_{2})>0$, and such that
\begin{equation}\label{one7.2}
T_{\rho}(1_{A_{i}}-1_{A_{j}})(u_{1})>T_{\rho}(1_{A_{i}}-1_{A_{j}})(u_{2})\quad\forall\, u_{1}\in U_{1},\, u_{2}\in U_{2}.
\end{equation}
Let $A\colonequals A_{i}$.  Let $\phi_{1},\phi_{2}\colon\R^{n}\to[0,\infty)$ be $C^{\infty}$ functions such that $(\partial A_{i})\cap(\partial A_{j})\cap\mathrm{support}(\phi_{1})$ is disjoint from $(\partial A_{i})\cap(\partial A_{j})\cap\mathrm{support}(\phi_{2})$, and such that
\begin{equation}\label{one7.3}
\int_{U_{1}}\phi_{1}(y)\gamma_{n}(y)dy=\int_{U_{2}}\phi_{2}(y)\gamma_{n}(y)dy.
\end{equation}
Define $V\colon\R^{n}\to\R^{n}$ by $V(y)\colonequals (\phi_{1}(y)-\phi_{2}(y))N$.  Note that $V$ is $C^{\infty}$, and
\begin{flalign*}
\int_{\partial A}\langle V(y),N(y)\rangle \gamma_{n}(y)dy
=\int_{U_{1}}\phi_{1}\gamma_{n}(y)dy-\int_{U_{2}}\phi_{2}\gamma_{n}(y)dy
\stackrel{\eqref{one7.3}}{=}0.
\end{flalign*}


So, $V$ satisfies \eqref{one7.1}.  We can therefore compute
\begin{flalign}
&\frac{d}{dt}|_{t=0}S_{\rho}(\{A_{p}^{(t)}\}_{p=1}^{n+1})
=\frac{d}{dt}|_{t=0}\int 1_{A_{i}^{(t)}}T_{\rho}1_{A_{i}^{(t)}}d\gamma_{n}+\frac{d}{dt}|_{t=0}\int 1_{A_{j}^{(t)}}T_{\rho}1_{A_{j}^{(t)}}d\gamma_{n}\nonumber\\
&\,\,\stackrel{\eqref{one7}}{=}2\int_{(\partial A_{i})\cap(\partial A_{j})}T_{\rho}1_{A_{i}}(y)\langle V(y),N(y)\rangle \gamma_{n}(y)dy
-2\int_{(\partial A_{i})\cap(\partial A_{j})}T_{\rho}1_{A_{j}}(y)\langle V(y),N(y)\rangle \gamma_{n}(y)dy\nonumber\\
&\,\,=2\int_{(\partial A_{i})\cap(\partial A_{j})}(T_{\rho}1_{A_{i}}(y)-T_{\rho}1_{A_{j}}(y))\langle V(y),N(y)\rangle \gamma_{n}(y)dy\nonumber\\
&\,\,=2\int_{(\partial A_{i})\cap(\partial A_{j})}(T_{\rho}1_{A_{i}}(y)-T_{\rho}1_{A_{j}}(y))(\phi_{1}(y)-\phi_{2}(y))\gamma_{n}(y)dy
\stackrel{\eqref{one7.2}\wedge\eqref{one7.3}}{>}0.\label{one7.7}
\end{flalign}

Finally, applying the Implicit Function Theorem as in \cite[Lemma 2.4]{barbosa84}, there exists $\epsilon>0$ such that, for every $t\in(-\epsilon,\epsilon)$, there exists a partition $\{\widetilde{A}_{p}^{(t)}\}_{p=1}^{n+1}$ of $\R^{n}$ with volumes $(a_{1},\ldots,a_{n+1})$ such that
\begin{equation}\label{one7.8}
\frac{d}{dt}|_{t=0}\sum_{p=1}^{n+1}\int_{\R^{n}}1_{A_{p}^{(t)}}T_{\rho}1_{A_{p}^{(t)}}d\gamma_{n}
=\frac{d}{dt}|_{t=0}\sum_{p=1}^{n+1}\int_{\R^{n}}1_{\widetilde{A}_{p}^{(t)}}T_{\rho}1_{\widetilde{A}_{p}^{(t)}}d\gamma_{n}.
\end{equation}
Then \eqref{one7.7} and \eqref{one7.8} imply \eqref{one1}.
\end{proof}

\section{Analyticity of the Ornstein-Uhlenbeck Operator}\label{secanalytic}


Let $A$ be a polyhedral cone, and let $\rho\in(-1,0)\cup(0,1)$.  In this section we show, roughly speaking, that if we restrict $T_{\rho}1_{A}$ to a line parallel to a facet of $A$, then we get a holomorphic function.  Let $\C=\{x_{1}+x_{2}\sqrt{-1}\colon x_{1},x_{2}\in\R\}$ denote the complex numbers.

The following statements establish some notation and assumptions on a partition $\{A_{p}\}_{p=1}^{n+1}$.  Rather than repeating these statements multiple times below, we will instead repeatedly reference them as Assumption \ref{as1}.
\begin{assumption}(\embolden{Geometric Assumptions})\label{as1}
Let $\{A_{p}\}_{p=1}^{n+1}$ be a flat partition of $\R^{n}$ with volumes $a\colonequals(a_{1},\ldots,a_{n+1})$.  Fix $i,j\in\{1,\ldots,n+1\}$, $i\neq j$ and let $\Pi\subset\R^{n}$ be a hyperplane so that $A_{i}\cap A_{j}\subset\Pi$.  Assume that $A_{i}$ and $A_{j}$ share a common facet.  Let $c\in\R$ and let $N\in\R^{n}$ with $\vnorm{N}_{2}=1$ so that
\begin{equation}\label{three1}
\Pi=\{x\in\R^{n}\colon\langle x,N\rangle=c\}.
\end{equation}
Without loss of generality, assume that
\begin{equation}\label{three2}
A_{i}\subset\{x\in\R^{n}\colon\langle x,N\rangle\leq c\}.
\end{equation}
Let $\mathrm{int}(A_{i}\cap A_{j})$ denote the relative interior of $A_{i}\cap A_{j}$.  Note that $A_{i}\cap A_{j}$ is a nonempty $(n-1)$-dimensional set.  Let $L\subset\Pi$ be any infinite line so that $L\cap\,\mathrm{int}(A_{i}\cap A_{j})\neq\emptyset$.  Let $w\in\R^{n}$ so that $L=\{cN+tw\colon t\in\R\}$.  Without loss of generality, assume $A_{i}\cap A_{j}\supset\{cN+tw\colon t\geq1\}$.
\end{assumption}

The following is the key lemma used to prove Theorem \ref{thm1}.

\begin{lemma}(\embolden{Analyticity Restricted to Lines})\label{lemma4}
Let $\rho\in(-1,0)\cup(0,1)$.  Suppose Assumption \ref{as1} holds.  Then the function $T_{\rho}(1_{A_{i}}-1_{A_{j}})$ satisfies the following two properties.
\begin{itemize}
\item[(a)] $\lim_{t\to\infty}T_{\rho}(1_{A_{i}}-1_{A_{j}})(cN+tw)=\lim_{t\to-\infty}T_{\rho}(1_{A_{i}}-1_{A_{j}})(cN+tw)$ if and only if $c=0$.
\item[(b)] The function $t\mapsto T_{\rho}(1_{A_{i}}-1_{A_{j}})(cN+tw)$, initially defined for $t\in\R$, can be extended to a holomorphic function of $t\in\C$.
\end{itemize}
\end{lemma}
\begin{proof}[Proof of (a)]  We first assume that $\rho\in(0,1)$.  By Assumption \ref{as1}, $A_{i}\cap A_{j}\supset\{cN+tw\colon t>1\}$.  Also, $L\cap\mathrm{int}(A_{i}\cap A_{j})\neq0$, and $A_{i},A_{j}$ are cones, so
\begin{flalign*}
&\lim_{t\to\infty}T_{\rho}(1_{A_{i}}-1_{A_{j}})(cN+tw)
=T_{\rho}(1_{\{z\in\R^{n}\colon\langle z,N\rangle\leq c\}}-1_{\{z\in\R^{n}\colon\langle z,N\rangle\geq c\}})(c)\\
&\qquad=\int 1_{(-\infty,c)}(c\rho+s\sqrt{1-\rho^{2}}\,)d\gamma_{1}(s)-\int 1_{(c,\infty)}(c\rho+s\sqrt{1-\rho^{2}}\,)d\gamma_{1}(s)\\
&\qquad=\int_{-\infty}^{(c-c\rho)/\sqrt{1-\rho^{2}}}d\gamma_{1}(s)-\int_{(c-c\rho)/\sqrt{1-\rho^{2}}}^{\infty}d\gamma_{1}(s)
=2\,\mathrm{sign}(c)\cdot\gamma_{1}([0,\abs{c}\frac{1-\rho}{\sqrt{1-\rho^{2}}})).
\end{flalign*}
So, $\lim_{t\to\infty}T_{\rho}(1_{A_{i}}-1_{A_{j}})(cN+tw)=0$ if and only if $c=0$.
In order to prove (a), it remains to show that $\lim_{t\to-\infty}T_{\rho}(cN+tw)=0$.  Since $cN+w\in\mathrm{int}(A_{i}\cap A_{j})$ by Assumption \ref{as1} and Definition \ref{flatdef}, we know that $\langle y_{i},w\rangle>\langle y_{p},w\rangle$ and $\langle y_{j},w\rangle>\langle y_{p},w\rangle$ for all $p\in\{1,\ldots,n+1\}\setminus\{i,j\}$.  So, $\langle y_{i},-w\rangle<\langle y_{p},-w\rangle$ and $\langle y_{j},-w\rangle<\langle y_{p},-w\rangle$ for all $p\in\{1,\ldots,n+1\}\setminus\{i,j\}$.  Therefore, there exists $\delta>0$ and $T<0$ such that, for all $t<T$, $B(cN + tw, \delta\abs{t})$ does not intersect either $A_i$ or $A_j$. Hence, there is some $\delta' > 0$ such that
\begin{equation}\label{zert}
\absf{\limsup_{t\to-\infty}T_{\rho}(1_{A_{i}}-1_{A_{j}})(cN+tw)}\leq\limsup_{t\to-\infty}\int_{B(\rho(c N+\abs{t}w),\delta' t(1-\rho^{2})^{-1/2})^{c}}d\gamma_{n}=0.
\end{equation}

The proof is therefore completed for $\rho\in(0,1)$.  For $\rho\in(-1,0)$, we similarly conclude that $\lim_{t\to-\infty}T_{\rho}(1_{A_{i}}-1_{A_{j}})(cN+tw)=0$ if and only if $c=0$, and $\lim_{t\to\infty}T_{\rho}(cN+tw)=0$.


\end{proof}
\begin{proof}[Proof of (b)]  We prove this by induction on $n$.  We first consider the case $n=2$.  In this case, $A_{p}$ is a two-dimensional sector for $p=1,2,3$.  Fix $i\in\{1,2,3\}$.  We will show that $T_{\rho}1_{A_{i}}(cn+tw)$ is holomorphic in $t$.  Without loss of generality, assume that one edge of $A_{i}$ is parallel to the $x_{2}$ axis.  Further, without loss of generality, suppose there exist $\alpha,\beta,\gamma\in\R$ such that $A_{i}=\{(x_{1},x_{2})\in\R^{2}\colon x_{1}\geq \alpha,x_{2}\geq \beta x_{1}+\gamma\}$.  Then
$$
T_{\rho}1_{A_{i}}(x)=\int_{(A_{i}-\rho x)/\sqrt{1-\rho^{2}}}d\gamma_{2}
=\int_{(\alpha-\rho x_{1})/\sqrt{1-\rho^{2}}}^{\infty}\int_{(\beta s+\gamma-\rho x_{2})/\sqrt{1-\rho^{2}}}^{\infty}e^{-r^{2}/2}e^{-s^{2}/2}drds/2\pi.
$$
Note that the set $\{cN+tw\colon t\geq1\}\subset\R^{2}$ has constant first coordinate, and the second coordinate of $\{cN+tw\colon t\geq1\}$ is an affine function of $t$.  So, defining a constant $\alpha'\colonequals(\alpha-\rho(cN+w)_{1})/\sqrt{1-\rho^{2}}\in\R$, we have
$$
T_{\rho}1_{A_{i}}(cN+tw)
=\int_{\alpha'}^{\infty}\int_{\beta' s+\gamma'+c't}^{\infty}e^{-r^{2}/2}e^{-s^{2}/2}drds/2\pi.
$$

For $s,t\in\R$, define
\begin{equation}\label{two1}
\phi(s,t)\colonequals
\int_{\beta' s+\gamma'+c't}^{\infty}e^{-r^{2}/2}dr/\sqrt{2\pi}
=\frac{1}{2}-\int_{0}^{\beta' s+\gamma'+c't}e^{-r^{2}/2}dr/\sqrt{2\pi}.
\end{equation}
For fixed $s\in\R$, $\phi(s,t)$ is a holomorphic function in $t$.  For $z\in\C$ and fixed $s\in\R$, define the following function, where we interpret the integral as a contour integral.
\begin{equation}\label{two2}
g(z)\colonequals\int_{0}^{\beta' s+\gamma'+c'z}e^{-r^{2}/2}dr/\sqrt{2\pi}.
\end{equation}
Since $e^{-r^{2}/2}$ is an entire function of $r\in\C$ with no poles, the choice of the path in the definition of $g$ does not matter.  In particular, for $z=(x_{1},x_{2})\in\C$, if we choose the path that first moves from $0$ to $(\beta' s+\gamma' +c'x_{1},0)$ along a straight line, and then from $(\beta' s+\gamma' +c'x_{1},0)$ to $(\beta' s+\gamma' +c'x_{1},c'x_{2})$ along a straight line, we get from \eqref{two2} the bound
\begin{equation}\label{two3}
\abs{g(z)}\leq \int_{0}^{\infty}e^{-r^{2}/2}\frac{dr}{\sqrt{2\pi}}+\absf{\int_{0}^{c'x_{2}}e^{-(\beta' s+\gamma'+c' x_{1}+ir)^{2}/2}dr}
\leq\frac{1}{2}+\int_{0}^{c'x_{2}}e^{r^{2}/2}dr
\leq\frac{1}{2}+c'\abs{x_{2}}e^{c'x_{2}^{2}/2}.
\end{equation}
In the penultimate estimate, we used $\absf{e^{-(x_{1}+ix_{2})^{2}}}\leq\absf{e^{x_{2}^{2}}}$.  Combining \eqref{two3} and \eqref{two1},
\begin{equation}\label{two4}
\abs{\phi(s,(x_{1},x_{2}))}\leq1+c'\abs{x_{2}}e^{c'x_{2}^{2}/2}.
\end{equation}

Consider the function $f_{\epsilon}(z)\colonequals\int_{\alpha}^{1/\epsilon}\phi(s,z)e^{-s^{2}/2}ds$, $z\in\C$.  This function is an entire function of $z\in\C$, by e.g. \cite[Theorem 5.4]{stein03}.  To prove that $f(z)\colonequals T_{\rho}1_{A_{i}}(cN+zw)$ is holomorphic, It now suffices to show that $f_{\epsilon}$ converges uniformly to $f$ on any compact subset of $\C$, as $\epsilon\to0$.  So, let $z=(x_{1},x_{2})\in\C$.  From \eqref{two4},
\begin{equation}\label{two5}
\begin{aligned}
\abs{f_{\epsilon}(z)-f(z)}
&\leq\int_{1/\epsilon}^{\infty}\abs{\phi(s,z)}e^{-s^{2}/2}ds\\
&\leq\int_{1/\epsilon}^{\infty}(1+c'\abs{x_{2}}e^{c'x_{2}^{2}/2})e^{-s^{2}/2}ds\leq(1+c'\abs{x_{2}}e^{c'x_{2}^{2}/2})e^{-1/(2\epsilon^{2})}.
\end{aligned}
\end{equation}
So, \eqref{two5} gives the desired uniform convergence property, completing the proof in the case $n=2$.  Note that, by \eqref{two4}
\begin{equation}\label{two6}
\abs{f(z)}\leq \int_{\alpha}^{\infty}\abs{\phi(s,z)}e^{-s^{2}/2}ds
\leq1+c'\abs{x_{2}}e^{c'x_{2}^{2}/2}
\end{equation}

For the more general case, we induct on $n$.  We are given an $n$-dimensional polyhedral cone $\Sigma_{n}$ with $n+1$ facets, and we write $\Sigma_{n}$ as a union of translates of an $(n-1)$-dimensional simplicial cone $\Sigma_{n-1}$, which can be taken to be one of the facets of $\Sigma_{n}$.  That is, there exists $u\in\R^{n}$ such that $\Sigma_{n}=\cup_{s>\alpha}(\Sigma_{n-1}+su)$.  So,
$$T_{\rho}1_{\Sigma_{n}}(cN+tw)=\int_{\alpha'}^{\infty}\int_{\Sigma_{n-1}+wt}d\gamma_{n-1}e^{-s^{2}/2}ds/\sqrt{2\pi}
\equalscolon\int_{\alpha'}^{\infty}\psi(s,t)e^{-s^{2}/2}ds/\sqrt{2\pi}.$$
Now, $\psi(s,t)$ is known, by the inductive hypothesis, to be holomorphic in $t=(x_{1},x_{2})$, for each $s$.  Moreover, by the inductive hypothesis, $\abs{\psi(s,(x_{1},x_{2}))}\leq1+c'\abs{x_{2}}e^{c'x_{2}^{2}/2}$.  We therefore truncate and make a tail estimate as in \eqref{two5} to prove uniform convergence of $f_{\epsilon}(z)\colonequals\int_{1/\epsilon}^{\infty}\Psi(s,z)e^{-s^{2}/2}ds$ to $f(z)\colonequals T_{\rho}1_{A_{i}}(cn+zw)$.  Finally, showing the following bound completes the inductive step.
$$\abs{f(z)}\leq\int_{\alpha}^{\infty}\abs{\Psi(s,z)}e^{-s^{2}/2}ds/\sqrt{2\pi}\leq 1+c'\abs{x_{2}}e^{c'x_{2}^{2}/2}.$$
\end{proof}
%

\section{The Main Theorem}\label{secmain}

Corollary \ref{cor4} below implies the Main Theorem \ref{thm1}.

\begin{cor}(\embolden{An Optimal Flat Partition Must be Centered})\label{cor5}
\begin{itemize}
\item[(1)]  Let $\rho\in(0,1)$.  Suppose $\{A_{i}\}_{i=1}^{n+1}$ is a partition of $\R^{n}$ satisfying
$$
S_{\rho}(\{A_{i}\}_{i=1}^{n+1})=
\sup_{\partsub}S_{\rho}(\{B_{i}\}_{i=1}^{n+1}).
$$
Suppose Assumption \ref{as1} holds.  Then $y=0$ in Assumption \ref{as1} (and in Definition \ref{flatdef}).
\item[(2)]   Let $\rho\in(-1,0)$.  Suppose $\{A_{i}\}_{i=1}^{n+1}$ is a partition of $\R^{n}$ satisfying
$$
S_{\rho}(\{A_{i}\}_{i=1}^{n+1})=
\inf_{\partsub}S_{\rho}(\{B_{i}\}_{i=1}^{n+1}).
$$
Suppose Assumption \ref{as1} holds.  Then $y=0$ in Assumption \ref{as1} (and in Definition \ref{flatdef}).
\end{itemize}
\end{cor}
\begin{proof}
Fix $i$ and $j$ in $\{1, \ldots, n+1\}$, and recall the definition of $c$ from Assumption~\ref{as1}. It suffices to show that if $c \neq 0$ then the partition $\{A_i\}_{i=1}^{n+1}$ is not optimal. Suppose, therefore, that $c \neq 0$. Then $T_{\rho}(1_{A_{i}}-1_{A_{j}})$ is not constant on any segment contained in $A_{i}\cap A_{j}$ by Lemma \ref{lemma4}. We then apply Lemma \ref{lemma1.1} to see that the partition is not optimal.
\end{proof}

\begin{cor}(\embolden{Biased Optimizers are not Flat})\label{cor4}
Let $\{A_{i}\}_{i=1}^{n+1}$ be a partition of $\R^{n}$ with volumes $a=(a_{1},\ldots,a_{n+1})$.
\begin{itemize}
\item[(1)]  Let $\rho\in(0,1)$.  Suppose $\{A_{i}\}_{i=1}^{n+1}$ is a partition of $\R^{n}$ satisfying
$$
S_{\rho}(\{A_{i}\}_{i=1}^{n+1})
=\sup_{\partsub}S_{\rho}(\{B_{i}\}_{i=1}^{n+1}).
$$
If Assumption \ref{as1} also holds, then we must have $a=(1/(n+1),\ldots,1/(n+1))$.
\item[(2)]   Let $\rho\in(-1,0)$.  Suppose $\{A_{i}\}_{i=1}^{n+1}$ is a partition of $\R^{n}$ satisfying
$$
S_{\rho}(\{A_{i}\}_{i=1}^{n+1})
=\inf_{\partsub}S_{\rho}(\{B_{i}\}_{i=1}^{n+1}).
$$
If Assumption \ref{as1} also holds, then we must have $a=(1/(n+1),\ldots,1/(n+1))$.
\end{itemize}
\end{cor}

\begin{proof}
Suppose that $\{A_i\}_{i=1}^{n+1}$ is a flat, optimal partition; we will show that it must be unbiased. First, Corollary~\ref{cor5} implies that $0 \in A_p$ for all $p \in \{1, \ldots, n+1\}$. Then, by Lemma \ref{lemma1.1} and \eqref{zert}, we have $T_{\rho}(1_{A_{i}}-1_{A_{j}})(0)=0$ for all $i, j \in \{1, \ldots, n+1\}$. Since each $A_p$ is a cone, its scale-invariance and the definition of $T_\rho$ imply that $\gamma_n(A_p) = (T_\rho 1_{A_p})(0)$; hence, $\gamma_{n}(A_{i})=\gamma_{n}(A_{j})$ for all $i,j\in\{1,\ldots,n+1\}$ with $i\neq j$.
\end{proof}

\section{Voting Interpretation}\label{secvoting}
As mentioned earlier our results imply that, generically speaking, plurality is not the most noise stable discrete function.
Gaussian space may be obtained as a limit of discrete spaces in multiple ways, each resulting in a somewhat different statement of Plurality is not stablest, via Corollary \ref{cor4}. For concreteness we prove Corollary~\ref{cor:plu}.
\begin{proof}[Proof of Cor. \ref{cor:plu}]
The proof follows by the Central Limit Theorem.  Let $f_{n}\colon\{1,2,3\}^{n}\to\{1,2,3\}$ be the sequence of plurality functions from Definition \ref{plurdef} and let $(\alpha,\beta)\neq(0,0)$.  In order to prove Corollary~\ref{cor:plu} we first claim that there exists a shifted flat partition $A=\{A_1,A_2,A_3\}$
of $\R^2$ such that
\[
\lim_{n \to \infty} \mathrm{P}[f_n = i] = \gamma_{2}(A_i), \forall\,i \in \{1,2,3\},
\]
and moreover
\[
\lim_{n \to \infty} \mathrm{P}[f_n(x) = f_n(y)] = \sum_{i=1}^3 \mathrm{P}[X \in A_i, Y \in A_i].
\]
This follows from the central limit theorem, noting that the vector $(X_i^{n})_{i=1}^3\in\R^{3}$ given by
\[
X_i^n(\omega_{1},\ldots,\omega_{n}) \colonequals n^{-1/2}  \sum_{j=1}^n (1(\omega_j = i) - 1/3),\,\,i=1,2,3
\]
converges to a centered normal vector $(N_i)_{i=1}^3\in\R^{3}$ with variances $2/9$ and covariances
$-2/3$ as $n\to\infty$. Moreover the partition $\{(X_i^{n})_{i=1}^3\colon f_n = 1\},\{(X_i^{n})_{i=1}^3\colon f_n = 2\}, \{(X_i^{n})_{i=1}^3\colon f_n = 3\}$ converges in the weak $L_{1}(\gamma_{2})$ norm to the sets given by
\[
A_i= \{ x\in\R^{2}:  i= \mathrm{argmax}_{j= 1,2,3} (N_1(x) + \alpha, N_2(x) + \beta, N_3(x) - \alpha - \beta) \}.
\]
In other words, $\{A_1,A_2,A_3\}$ is a shifted standard simplex.

Let $\{B_1,B_2,B_3\}$ be a partition of $\R^2$ which satisfies $\gamma_{2}(A_i) = \gamma_{2}(B_i)$ for all $i=1,2,3$ and
\[
\sum_{i=1}^3 \mathrm{P}[X \in A_i, Y \in A_i] < \sum_{i=1}^3 \mathrm{P}[X \in B_i, Y \in B_i].
\]
Note that $\{B_1,B_2,B_3\}$ exists by Theorem \ref{thm1}.  By approximating $\{B_{1},B_{2},B_{3}\}$ by a finite number of axis-parallel rectangles, we may assume that $\{B_{1},B_{2},B_{3}\}$ consists of a finite union of axis-parallel rectangles.

Consider now $\tilde{g}_n\colon\{1,2,3\}^{n}\to\{1,2,3\}$ satisfying $\tilde{g}_n = i$ if $(X_1^n, X_2^n, X_3^n) \in B_i$.  Since $\{B_{1},B_{2},B_{3}\}$ consists of a finite number of axis-parallel rectangles, Definition \ref{infdef} shows that $\max_{i=1,\ldots,n}I_{i}\tilde{g}_{n} = O(n^{-1/2})$.  Moreover by the Central Limit Theorem, we see that
\[
|\mathrm{P}[\tilde{g}_n = i] - \gamma_{2}(B_i)| = O(n^{-1/2}), \quad
|\mathrm{P}[\tilde{f}_n = i] - \gamma_{2}(A_i)| = O(n^{-1/2}).
\]
and
\[
\lim_{n \to \infty} \mathrm{P}[\tilde{g}_n(x) = \tilde{g}_n(y)] = \sum_{i=1}^3 \mathrm{P}[X \in B_i, Y \in B_i] >
\sum_{i=1}^3 \mathrm{P}[X \in A_i, Y \in A_i]
\]
as needed.

%
\end{proof}

\medskip

\noindent{\textbf{Acknowledgement.}} Thanks to Oded Regev for reading and commenting on the manuscript.  Thanks to Ronen Eldan, Kostya Makarychev, and Assaf Naor for helpful comments. This work was mostly conducted during the semester on ``Real Analysis in Computer Science'' at the Simons Institute for Theoretical Computer Science. We are grateful to the institute for its hospitality and support.

\bibliographystyle{abbrv}
\bibliography{12162011}

\end{document}